\documentclass[12pt,a4paper,twoside,reqno]{amsart}

\usepackage{amsmath, amssymb, amsthm, enumerate, enumitem}
\usepackage[normalem]{ulem}

\newcommand{\MM}{\mathcal{M}}
\newcommand{\CC}{\mathcal{C}}
\newcommand{\IB}{\mathbb{B}}
\newcommand{\IR}{\mathbb{R}}

\newcommand{\IC}{\mathbb{C}}

\newcommand{\IZ}{\mathbb{Z}}

\newcommand{\ov}[1]{\overline{#1}}

\DeclareMathOperator{\Klein}{Klein}
\DeclareMathOperator{\thick}{thick}
\DeclareMathOperator{\thin}{thin}
\DeclareMathOperator{\eucl}{eucl}
\DeclareMathOperator{\stan}{stan}
\DeclareMathOperator{\scal}{scal}

\DeclareMathOperator{\hyp}{hyp}
\DeclareMathOperator{\Ric}{Ric}

\DeclareMathOperator{\area}{area}

\DeclareMathOperator{\tr}{tr}

\DeclareMathOperator{\dist}{dist}

\DeclareMathOperator{\diam}{diam}

\DeclareMathOperator{\vol}{vol}

\DeclareMathOperator{\Rm}{Rm}
\newcommand{\td}[1]{\widetilde{#1}}
\newcommand{\cangle}{\widetilde{\sphericalangle}}

\newcommand{\EMPTY}[1]{}

\newtheorem{Theorem}{Theorem}[section]
\newtheorem{Lemma}[Theorem]{Lemma}

\newtheorem{Proposition}[Theorem]{Proposition}
\newtheorem{Definition}[Theorem]{Definition}
\numberwithin{equation}{section}

\title{Long-time analysis of 3 dimensional Ricci flow I}
\author{Richard H Bamler}
\address{Stanford University, Department of Mathematics, 450 Serra Mall, building 380, Stanford, California 94305}
\email{rbamler@stanford.edu}
\date{\today}

\begin{document}
\begin{abstract}
In this paper we analyze the long-time behaviour of $3$ dimensional Ricci flow with surgery.
We prove that under the topological condition that the initial manifold only has non-aspherical or hyperbolic components in its geometric decomposition, there are only finitely many surgeries and the curvature is bounded by $C t^{-1}$ for large $t$.
This answers an open question in Perelman's work, which was made more precise by Lott and Tian, for this class of initial topologies.

More general classes will be discussed in subsequent papers using similar methods.
\end{abstract}

\maketitle
\tableofcontents

\section{Introduction} \label{sec:Introduction}
This is the first of a series of papers in which we analyze the long-time behavior of the Ricci flow with surgery on $3$ dimensional manifolds.
The main result of this paper will be the following theorem, which we will present more precisely at the end of the introduction:

\begin{quote}
\textit{Let $(M,g)$ be a closed $3$ dimensional Riemannian manifold which fulfills the pure topological condition that all components of its geometric decomposition are hyperbolic or non-aspherical. \\
Then there is a long-time existent Ricci flow which has only \emph{finitely} many surgeries and whose initial metric is $g$.
Moreover, the Riemannian curvature in this flow is bounded everywhere by $C t^{-1}$ for large $t$.}
\end{quote}

The Ricci flow with surgery has been used by Perelman to solve the Poincar\'e and Geometrization Conjecture (\cite{PerelmanI}, \cite{PerelmanII}, \cite{PerelmanIII}).
More precisely, given any initial metric on a closed $3$-manifold, Perelman managed to construct a solution for the Ricci flow with surgery on a maximal time interval and showed that the surgery times do not accumulate.
Hence in every finite time interval there are only a finite number of surgery times.
Furthermore, he could prove that if the given manifold is a homotopy sphere (or more generally a connected sum of prime, non-aspherical manifolds), then the Ricci flow goes extinct in finite time.
This implies that the initial manifold is a sphere if it is simply connected and hence establishes the Poincar\'e Conjecture.
On the other hand, if the Ricci flow continues to exist, he could show that the manifold decomposes into a thick part which approaches a hyperbolic metric and an thin part which becomes arbitrarily collapsed on local scales.
Based on this collapsing, it is then possible to show that the thin part can be decomposed into geometric pieces (\cite{ShioyaYamaguchi}, \cite{MorganTian}, \cite{KLcollapse}) and hereby establish the Geometrization Conjecture.

Observe that although the Ricci flow with surgery was used to solve such difficult problems, some of its basic properties are still unknown, because they surprisingly turned out to be irrelevant in the end.
For example, the question remains whether in the long-time existent case there are finitely many surgery times, i.e. whether after some time the flow can be continued by a conventional smooth, nonsingular Ricci flow defined up to time infinity.
Furthermore, it is still unknown whether and in what way the Ricci flow exhibits the the full geometric decomposition of the manifold.
These questions follow naturally from Perelman's work and are partially explicitly raised there.
It has been conjectured Tian and Lott that they can be answered positively.

In \cite{LottTypeIII}, \cite{LottDimRed} and \cite{LottSesum}, Lott and Lott-Sesum could give a description of the long-time behaviour of certain Ricci flows on manifolds which consist of a single component in their geometric decomposition.
However, they needed to make additional curvature and diameter or symmetry assumptions.

In this paper, we only have to impose a topological condition on the initial manifold.
Using the language developed in section \ref{sec:DefRFsurg} our precise result reads:
\begin{Theorem} \label{Thm:main}
Given a surgery model $(M_{\stan}, g_{\stan}, D_{\stan})$, there is a continuous function $\delta : [0, \infty) \to \IR_+$ such that.

Let $\MM$ be a Ricci flow with surgeries with normalized initial conditions and $\delta(t)$-precise cutoff (see section \ref{sec:DefRFsurg} for more details) such that $\MM(0)$ satisfies the following topological condition: 

$\MM(0) \approx M_1 \# \ldots \# M_m$ is a connected sum of closed $3$-manifolds $M_i$.
Each $M_i$ is either spherical, diffeomorphic to $S^2 \times S^1$ or its torus decomposition only consists of hyperbolic pieces (i.e. we can find collections of pairwise disjoint, incompressible, embedded tori $T_{i,1}, \ldots, T_{i,m_i} \subset M_i$ such that the connected components of $M \setminus (T_{i,1} \cup \ldots \cup T_{i,m_i})$ carry complete finite volume hyperbolic metrics).

Then $\MM$ has only finitely many surgeries and there are constants $T, C < \infty$ such that $|{\Rm}| < C t^{-1}$ on $\MM(t)$ for all $t \geq T$.
\end{Theorem}

We like to point out that up to now, the curvature estimate was only known to hold on the thick part.
Hence, our theorem contributes towards a better understanding of the geometry of the thin part.
Observe that the result implies that the rescaled metrics $t^{-1} g(t)$ have uniformly bounded curvature for large $t$.
Such solutions are said to be of type III and have been subject of study by Hamilton (\cite{Ham}).

We give a short outline of the proof:
The thin part of the manifold is locally collapsed along $S^1$, $T^2$ or $S^2$ fibers.
We will show that there are certain ``good'' areas where the fibers are either diffeomorphic to $S^1$ or $T^2$ and incompressible in the manifold.
Hence, if we pass to the universal cover, these areas will become non-collapsed on a local scale.
We can then use a modification of Perelman's Theorem \cite[7.3]{PerelmanII} to deduce a curvature bound on the scale $\sqrt{t}$.
By looking closer at the decomposition arising from the collapse, we can argue that if not all areas of the thin part are good, there must be some good area which is collapsed along incompressible $S^1$-fibers over a $2$-dimensional space.
Hence, by the conclusion above, this collapse takes place at scale $\sqrt{t}$.
Next, we establish the existence of minimal annuli which intersect every fiber of this fibration and whose area goes to zero compared to the scale $\sqrt{t}$.
This will then give us a contradiction implying that the thin part only consists of good areas and hence the curvature is controlled everywhere.

The paper is organized as follows:
In section \ref{sec:DefRFsurg} we clarify the concepts behind Ricci flows with surgery.
We are keeping the definitions here as general as possible so that they match or follow from existent literature on the subject.
Section \ref{sec:ExRFsurg} recapitulates known existence results for Ricci flows with surgery.
In section \ref{sec:Perelman} we quote Perelman's important long-time curvature estimate and generalize it to the universal cover.
We then explain the known geometric results arising from the long-time analysis in sections \ref{sec:thickthin} and \ref{sec:MorganTian}.
In section \ref{sec:unwrap} we analyze the behaviour of the collapse when passing to the universal cover and in section \ref{sec:minsurf} we prove bounds for the evolution of minimal spheres and annuli in Ricci flow.
Finally, the proof of the main theorem can be found in section \ref{sec:MainThm}.

I would like to thank Gang Tian for his constant help and encouragement and John Lott for many long conversations.
I am also indebted to Bernhard Leeb and Hans-Joachim Hein, who contributed essentially to my understanding of Perelman's work.
Thanks also go to Simon Brendle, Daniel Faessler, Robert Kremser, Tobias Marxen, Rafe Mazzeo, Richard Schoen, Stephan Stadler and Brian White.

\section{Definition of Ricci flows with surgery} \label{sec:DefRFsurg}
In this section, we give a precise definition of the Ricci flows with surgery that we are going to analyze.
We will mainly use the language developed in \cite{Bamler-diploma} here.
We first define Ricci flows with surgery in a very broad sense
\begin{Definition}[Ricci flow with surgery]
Consider a time interval $I \subset \IR$.
Let $T^1 < T^2 < \ldots$ be times of the interior of $I$ which form a possibly infinite, but discrete subset of $\IR$ and divide $I$ into the intervals
\[ I^1 = I \cap (-\infty, T^1), \quad I^2 = [T^1, T^2), \quad I^3 = [T^2, T^3), \quad \ldots \]
and $I^{k+1} = I \cap [T^k,\infty)$ if there are only finitely many $T^i$'s.
Consider Ricci flows $(M^1 \times I^1, g^1_t), (M^2 \times I^2, g^2_t), \ldots$ on $3$-manifolds $M^1, M^2, \ldots$.
Let $\Omega^i \subset M^i$ be open sets on which the metric $g^i_t$ converges smoothly as $t \nearrow T^i$ to some Riemannian metric $g_{T^i}$ on $\Omega_i$ and let 
\[ U^i_- \subset \Omega^i \qquad \text{and} \qquad U^i_+ \subset M^{i+1} \]
be open subsets such that there are isometries
\[ \Phi^i : (U^i_-, g_{T^i}^i) \longrightarrow (U^i_+, g_{T^i}^{i+1}), \qquad (\Phi^i)^* g_{T^i}^{i+1} |_{U_+^i} = g^i_{T^i} |_{U^i_-}. \]
We assume that we never have $U_-^i = \Omega^i = M^i$ and $U_+^i = M^{i+1}$ and that every component of $M^{i+1}$ contains a point of $U^i_+$.
Then, we call $\MM = ((T^i)_i, (M^i \times I^i, g_t^i)_i, (\Omega^i)_i, (U^i_{\pm})_i, (\Phi^i)_i)$ a \emph{Ricci flow with surgery on the time interval $I$} and $T^1, T^2, \ldots$ \emph{surgery times}.

If $t \in I^i$, then $(\MM(t), g(t)) = (M^i \times \{ t \}, g_t^i)$ is called the \emph{time $t$-slice of $\MM$}.
For $t = T^i$, we define the \emph{(presurgery) time $T^{i-}$-slice} to be $(\MM(T^{i-}), g(T^{i-})) = (\Omega^i \times \{ T^i \}, g^i_{T^i})$.
The points $\Omega^i \times \{ T^i \} \setminus U^i_- \times \{ T^i \}$ are called \emph{presurgery points} and the points $M^{i+1} \times \{ T^i \} \setminus U^i_+ \times \{ T^i \}$ \emph{postsurgery points}.
We will call a point that is not a presurgery point a \emph{non-presurgery point}.
\end{Definition}

We will make use of the following vocabulary when dealing with Ricci flows with surgery:

\begin{Definition}[Ricci flow with surgery, points in time] \label{Def:pointsurvives}
For $(x,t) \in \MM$, consider a spatially constant line in $\MM$ that starts in $(x,t)$ and goes forward or backward in time for some time $\Delta t \in \IR$ and that doesn't hit any (pre- or \hbox{post-)}surgery points except possibly at its endpoints.
When crossing surgery times, we can continue the line via the isometries $\Phi^i$.
We denote the endpoint of this line by $(x,t+\Delta t) \in \MM$.
Observe that this point is only defined if there are no surgery points between $(x,t)$ and $(x,t+\Delta t)$.
We say that a point $(x,t) \in \MM$ \emph{survives until time} $t + \Delta t$ if the point $(x,t+\Delta t) \in \MM$ is well-defined.

Observe that this notion also makes sense, if $(x, t^-) \in \MM$ is a presurgery point and $\Delta t \leq 0$.
\end{Definition}

Using this definition, we can define parabolic neighborhoods in $\MM$.

\begin{Definition}[Ricci flow with surgery, parabolic neighborhoods]
Let $(x,t) \in \MM$ (presurgery points are allowed, in this case we have to replace $t$ by $t^-$), $r \geq 0$ and $\Delta t \in \IR$.
Consider the ball $B = B(x,t,r) \subset \MM(t)$.
If $(x,t^-)$ is a presurgery point, we have to look at $B(x,t^-,r) \subset \MM(t^-)$.
For each $(x',t) \in B$ consider the union $I^{\Delta t}_{x',t}$ of all points $(x',t+t') \in \MM$ which are well-defined in the sense of Definition \ref{Def:pointsurvives} for $t' \in [0, \Delta t]$ resp. $t' \in [\Delta t, 0]$.
We say that $I^{\Delta t}_{x',t}$ is \emph{non-singular} if $(x', t + \Delta t) \in I^{\Delta t}_{x',t}$.
Define the \emph{parabolic neighborhood} $P(x,t,r,\Delta t) = \bigcup_{x' \in B} I^{\Delta t}_{x',t}$.
We call $P(x,t,r,\Delta t)$ \emph{non-singular} if all the $I^{\Delta t}_{x',t}$ are non-singular.
\end{Definition}

We will now characterize three important approximate local geometries that we will have to deal with very often: $\varepsilon$-necks, strong $\varepsilon$-necks and $(\varepsilon, E)$-caps.
The notions below also make sense for presurgery times.

\begin{Definition}[Ricci flow with surgery, $\varepsilon$-necks]
Let $\varepsilon > 0$ and consider a Ricci flow with surgery $\MM$.
Let $t$ be a time of $\MM$.
We call an open subset $U \subset \MM(t)$ an $\varepsilon$-neck, if there is a smooth bijective map $\Phi : S^2 \times (-\frac1{\varepsilon}, \frac1{\varepsilon}) \to U$ such that there is a $\lambda > 0$ with $\Vert \lambda^{-2} \Phi^* g(t) - g_{S^2 \times \IR} \Vert_{C^{[\varepsilon^{-1}]}} < \varepsilon$ where $g_{S^2 \times \IR}$ is the standard metric on $S^2 \times (-\frac1{\varepsilon}, \frac1{\varepsilon})$.

We say that $x \in \MM(t)$ is a \emph{center} of $U$ if $x \in \Phi(S^2 \times \{0\})$ for such a $\Phi$.
\end{Definition}

\begin{Definition}[Ricci flow with surgery, strong $\varepsilon$-necks]
Let $\varepsilon > 0$ and consider a Ricci flow with surgery $\MM$ and a time $t_2$.
Consider a subset $U \subset \MM(t_2)$ and assume that all points of $U$ survive until some time $t_1 <  t_2$.
Then the subset $U \times [t_1,t_2] \subset \MM$ is called a \emph{strong $\varepsilon$-neck} if there is a factor $\lambda > 0$ such that after parabolically rescaling by $\lambda^{-1}$, the flow on $U \times [t_1,t_2]$ is $\varepsilon$-close to the standard flow on $[-1,0]$.
By this we mean $\lambda^{-2} (t_2 - t_1) = 1$ and there is a diffeomorphism $\Phi : S^2 \times (-\frac1{\varepsilon}, \frac1{\varepsilon}) \to U$ such that for all $t \in [t_1, t_2]$
\[ \Vert \lambda^{-2} \Phi^* g(t) - g_{S^2 \times \IR} (\lambda^{-2} (t-t_2)) \Vert_{C^{[\varepsilon^{-1}]}} < \varepsilon. \]
Here $(g_{S^2 \times \IR}(t))_{t \in (-\infty,0]}$ is the standard Ricci flow on $S^2 \times \IR$ which has scalar curvature $1$ at time $0$.
\end{Definition}

\begin{Definition}[Ricci flow with surgery, $(\varepsilon, E)$-caps]
Let $\varepsilon, E > 0$ and consider a Ricci flow with surgery $\MM$.
Let $t$ be a time of $\MM$ and $x \in \MM(t)$.
Consider an open set $U \subset \MM(t)$ and suppose that $(\diam_t U)^2 |{\Rm}|(y,t) < E^2$ for any $y \in U$ and $E^{-2} |{\Rm}|(y_1,t) \leq |{\Rm}|(y_2,t) \leq E^2 |{\Rm}|(y_1,t)$ for any $y_1, y_2 \in U$.
Furthermore, assume that $U$ is either diffeomorphic to $\IB^3$ or $\IR P^3 \setminus \ov{\IB}^3$ and that there is a compact set $K \subset U$ such that $U \setminus K$ is an $\varepsilon$-neck.

Then $U$ is called an \emph{$(\varepsilon, E)$-cap}.
If $x \in K$ for such a $K$, then we say that $x$ is a center of $U$.
\end{Definition}

With these concepts at hand we can now give an exact description of the surgery process that will underlie the Ricci flows with surgeries which we are going to analyze.
The author has chosen the phrasing so that it includes the outcomes of the constructions presented in \cite{PerelmanII}, \cite{KLnotes}, \cite{MTRicciflow}, \cite{BBBMP} and \cite{Bamler-diploma}.

We will first need to fix a geometry which models the metric which we will endow the filling $3$-balls with after each surgery.
\begin{Definition}[surgery model]
Consider $M_{\stan} = \IR^3$ with its natural $SO(3)$-action and let $g_{\stan}$ be a complete metric on $M_{\stan}$ such that
\begin{enumerate}
\item $g_{\stan}$ is $SO(3)$-invariant,
\item $g_{\stan}$ has nonnegative sectional curvature,
\item for any sequence $x_n \in M_{\stan}$ with $\dist(0, x_n) \to \infty$, the pointed Riemannian manifolds $(M_{\stan}, g_{\stan}, x_n)$ smoothly converge to the standard $S^2 \times \IR$ of scalar curvature $R = 1$.
\end{enumerate}
For every $r > 0$, we denote the $r$-ball around $0$ by $M_{\stan}(r)$.
Let $D_{\stan} > 0$ be a positive number.
Then we call $(M_{\stan}, g_{\stan}, D_{\stan})$ a \emph{surgery model}.
\end{Definition}

\begin{Definition}[$\varphi$-positive curvature]
We say that a Riemannian metric $g$ on a manifold $M$ has \emph{$\varphi$-positive curvature} for $\varphi > 0$ if for every point $p \in M$ there is an $X > 0$ such that $\sec_p \geq - X$ and
\[ \scal_p \geq - \tfrac32 \varphi \qquad \text{and} \qquad \scal_p \geq 2 X (\log (2 X) - \log \varphi - 3). \]
\end{Definition}
Observe that by \cite{Ham} this condition is improved by Ricci flow in the following sense: If $(M, (g_t)_{t \in [t_0, t_1]})$ is a Ricci flow with $t_0 > 0$ and $g_{t_0}$ is $t_0^{-1}$-positive, then $g_t$ is $t^{-1}$-positive for all $t \in [t_0, t_1]$.

\begin{Definition}[Ricci flow with surgery, $\delta(t)$-precise cutoff] \label{Def:precisecutoff}
Let $\MM$ be a Ricci flow with surgery defined on some time interval $[0,T)$, let $(M_{\stan},g_{\stan},D_{\stan})$ be a surgery model and let $\delta : [0,\infty) \to \IR_+$ be a function.
We say that $\MM$ is \emph{performed by $\delta(t)$-precise cutoff (using the surgery model $(M_{\stan},g_{\stan}, D_{\stan})$)} if
\begin{enumerate}
\item For all $t$ the metric $g(t)$ (and $g(t^-)$ if $t$ is a surgery time) has $t^{-1}$-positive curvature.
\item For every surgery time $T^i$, the subset $\MM(T^i) \setminus U^i_+$ is a disjoint union $D^i_1 \cup \ldots \cup D^i_{m_i}$ of smoothly embedded $3$-disks.
\item For every such $D^i_j$ there is an embedding 
\[ \Phi^i_j : M_{\stan}(\delta^{-1}(T^i)) \longrightarrow \MM(T^i) \]
such that $D^i_j \subset \Phi^i_j (M_{\stan}(D_{\stan}))$ and such that for all $j = 1, \ldots, m_i$ the images $\Phi^i_j (M_{\stan}(\delta^{-1}(T^i)))$ are pairwise disjoint and there is are constants $0 <\lambda^i_j \leq \delta(T^i)$ such that 
\[ \big\Vert g_{\stan} - (\lambda^i_j)^{-2} (\Phi^i_j)^* g(T^i) \big\Vert_{C^{[\delta^{-1}(T^i)]}(M_{\stan}(\delta^{-1}(T^i)))} < \delta(T^i). \]
\item For every such $D^i_j$, the points on the boundary of $U^i_-$ in $\MM(T^{i-})$ corresponding to $\partial D^i_j$ are centers of strong $\delta(T^i)$-necks.
\item For every $D^i_j$ for which the boundary component of $\partial U^i$ corresponding to the sphere $\partial D^i_j$ bounds a $3$-disk component $(D')^i_j$ of $M^i \setminus U^i$ (i.e. a ``trivial surgery'', see below), the following holds:
For every $\chi > 0$, there is some $t_\chi < T^i$ such that for all $t \in (t_\chi,T^i)$ there is a $(1+\chi)$-Lipschitz map $\xi : (D')^i_j \to D^i_j$ which corresponds to the identity on the boundary.
\item For every surgery time $T^i$, the components of $\MM(T^{i-}) \setminus U^i_-$ are either diffeomorphic to $S^2 \times I$, $D^3$, $\IR P^3 \setminus B^3$, a spherical space form, $S^1 \times S^2$ or $\IR P^3 \# \IR P^3$.
\end{enumerate}
We will speak of each $D^i_j$ as \emph{a surgery} and if $D^i_j$ satisfies the property described in (5), we call it a \emph{trivial surgery}.
\end{Definition}
Observe that we have phrased the Definition so that if $\MM$ is a Ricci flow with surgery which is performed by $\delta(t)$-precise cutoff, it is also performed by $\delta'(t)$-precise cutoff whenever $\delta'(t) \geq \delta(t)$ for all $t$.
Note also that trivial surgeries don't change the topology of the respective component at which they are performed.

\section{Existence of Ricci flows with surgery} \label{sec:ExRFsurg}
Ricci flows with surgery and precise cutoff as introduced in Definition \ref{Def:precisecutoff} can indeed be constructed from any given initial metric.
We will make this more precise below.
To simplify things, we restrict the geometries which we want to consider as initial conditions.

\begin{Definition}[Normalized initial conditions]
We say that a Riemannian $3$-manifold $(M,g)$ is \emph{normalized} if 
\begin{enumerate}
\item $M$ is compact and orientable,
\item $|{\Rm}| < 1$ everywhere and
\item $\vol B(x,1) > \frac{\omega_3}2$ for all $x \in M$ where $\omega_3$ is the volume of a standard Euclidean $3$-ball.
\end{enumerate}
We say that a Ricci flow with surgery $\MM$ has \emph{normalized initial conditions}, if $\MM(0)$ is normalized.
\end{Definition}
Obviously, any Riemannian metric on a compact and orientable $3$-manifold can be rescaled to be normalized.
Moreover, recall
\begin{Definition}[$\kappa$-noncollapsedness]
Let $\MM$ be a Ricci flow with surgery, $(x,t) \in \MM$ (possibly a presurgery point) and $\kappa, \rho > 0$.
We say that $\MM$ is $\kappa$-noncollapsed in $(x,t)$ on scales less than $\rho$ if $\vol_t B(x,t,r) \geq \kappa r^3$ for all $0 < r < \rho$ for which
\begin{enumerate}
\item the ball $B(x,t,r)$ is relatively compact in $\MM(t)$,
\item the parabolic neighborhood $P(x,t,r, -r^2)$ is nonsingular and
\item $\Vert \Rm \Vert < r^{-2}$ on $P(x,t,r,-r^2)$.
\end{enumerate}
\end{Definition}

In order to construct a Ricci flow with surgery, we need the following characterization of regions of high curvature (see \cite{PerelmanII}, \cite{KLnotes}, \cite{MTRicciflow}, \cite{BBBMP}, \cite{Bamler-diploma}).
The power of the this proposition lies in the fact that none of the parameters depends on the number or the preciseness of the preceding surgeries.
Hence, it provides a tool to perform surgeries in a controlled way.
\begin{Proposition}[Canonical neighborhood theorem, Ricci flows with surgery] \label{Prop:CNThm-mostgeneral}
There are constants $C_0 < \infty$ and $\kappa_0 > 0$ and for every surgery model $(M_{\stan}, \linebreak[1] g_{\stan}, \linebreak[1] D_{\stan})$ and every $\varepsilon > 0$ there are a constant $E < \infty$ and continuous positive functions $r, \delta, \kappa : [0,\infty) \to \IR_+$ such that the following holds:

Let $\MM$ be a Ricci flow with surgery on some time interval $[0,T)$ which has normalized initial conditions and which is performed by $\delta(t)$-precise cutoff.
Then
\begin{enumerate}[label=(\textit{\alph*})]
\item At every time $t \in [0,T)$ the flow $\MM$ is $\kappa(t)$-noncollapsed on scale less than $\sqrt{t}$.
\item If $(x,t) \in \MM$ is a non-presurgery point with $R(x,t) \geq r^{-2}(t)$, then
\begin{enumerate}[label=(\textit{\arabic*})]
\item $(x,t)$ is either the center of a strong $\varepsilon$-neck or an $(\varepsilon, E)$-cap,
\item $\Vert \nabla R^{-1/2} (x,t) \Vert < C_0$ and $| \partial_t R^{-1}(x,t) | < C_0$,
\item $\MM$ is $\kappa_0$-noncollapsed in $(x,t)$.
\end{enumerate}
\end{enumerate}
\end{Proposition}
Using Proposition \ref{Prop:CNThm-mostgeneral}, it is possible to give an existence result for Ricci flows with surgery.
For a proof see again the sources indicated above.
\begin{Proposition} \label{Prop:RFwsurg-existence}
Given a surgery model $(M_{\stan}, g_{\stan}, D_{\stan})$, there is a continuous function $\delta : [0, \infty) \to \IR_+$ such that if $\delta' : [0, \infty) \to \IR_+$ is a continuous function with $\delta'(t) \leq \delta(t)$ for all $t \in [0,\infty)$ and $(M,g)$ is a normalized Riemannian manifold, then there is a Ricci flow with surgery $\MM$ defined for times $[0, \infty)$ such that $\MM(0) = (M,g)$ and which is performed by $\delta'(t)$-precise cutoff. (Observe that we can possibly have $\MM(t) = \emptyset$ for large $t$.)

Moreover, if $\MM$ is a Ricci flow with surgery on some time interval $[0,T)$ which has normalized initial conditions and which is performed by $\delta(t)$-precise cutoff, then $\MM$ can be extended to a Ricci flow on the time interval $[0, \infty)$ which has $\delta'(t)$-precise cutoff on the time interval $[T, \infty)$.
\end{Proposition}

We point out that the parameters $\delta(t)$ and $\varepsilon$ in Proposition \ref{Prop:CNThm-mostgeneral} and $\delta(t)$ in Proposition \ref{Prop:RFwsurg-existence} depend on the choice of the surgery model.
\begin{quote}
\textit{\textbf{From now on we will fix a surgery model $(M_{\stan}, g_{\stan}, D_{\stan})$ for the rest of this paper and we will not mention this dependence anymore.}}
\end{quote}

\section{Perelman's longtime analysis result} \label{sec:Perelman}
Consider a Ricci flow with surgery $\MM$.
For any non-presurgery point $(x,t) \in \MM$, we define
\[ \rho (x,t) = \max \{ r > 0 \; \; : \; \; \sec \geq - r^{-2} \quad \text{on} \quad B(x,t,r) \}. \]

The following Proposition is a consequence of \cite[6.8, 7.3]{PerelmanII}:
\begin{Proposition} \label{Prop:Per73}
There is a continuous positive function $\delta : [0, \infty) \to \IR_+$ such that for every $w > 0$ there are constants $\ov{\rho}(w), \ov{r}(w) > 0$ and $T = T(w), K = K(w) < \infty$ such that:

Let $\MM$ be a Ricci flow with surgery on the time interval $[0, \infty)$ with normalized initial conditions which is performed by $\delta(t)$-precise cutoff.
Let $(x,t) \in \MM$ be a non-presurgery point with $t > T$.
\begin{enumerate}[label=(\textit{\alph*})]
\item If $0 < r \leq \min \{ \rho(x,t), \ov{r} \sqrt{t} \}$ and $\vol_t B(x,t, r) \geq w r^3$, then $|{\Rm}| < K r^{-2}$ on $B(x, t, r)$.
\item If $\vol_t B(x,t,\rho(x,t)) \geq w \rho^3 (x,t)$, then $\rho(x,t) > \ov{\rho} \sqrt{t}$ and $|{\Rm}| < K t^{-1}$ on $B(x,t,\ov{\rho}\sqrt{t})$.
\end{enumerate}
\end{Proposition}
We can generalize this Proposition by passing to the universal cover:
Consider a non-presurgery point $(x,t) \in \MM$ and $r > 0$.
Lift $x \in \MM(t)$ to the universal cover $\widetilde{\MM}(t)$ of $\MM(t)$ to obtain $\widetilde{x}$.
Then we call $\vol_t \td{B}(\widetilde{x},t, r)$ the volume of the $r$-ball around $\widetilde{x}$ in $\widetilde{\MM}(t)$.
Obviously, $\vol_t \widetilde{B}(\widetilde{x},t, r) \geq \vol_t B(x,t,r)$.

We can now state the following more general Proposition which will be crucial for the proof of Theorem \ref{Thm:main}:
\begin{Proposition} \label{Prop:Per73univcover}
Under the same assumptions as in Proposition \ref{Prop:Per73}, we have:
\begin{enumerate}[label=(\textit{\alph*})]
\item If $0 < r \leq \min \{ \rho(x,t), \ov{r} \sqrt{t} \}$ and $\vol_t \td{B}(\td{x},t, r) \geq w r^3$, then $|{\Rm}| < K r^{-2}$ on $B(x, t, r)$.
\item If $\vol_t \td{B}(\td{x},t,\rho(x,t)) \geq w \rho^3 (x,t)$, then $\rho(x,t) > \ov{\rho} \sqrt{t}$ and $|{\Rm}| < K t^{-1}$ on $B(x,t,\ov{\rho}\sqrt{t})$.
\end{enumerate}
\end{Proposition}
\begin{proof}
We first need to define the universal covering flow $\widetilde{\MM}$ of $\MM$.
Recall that $\MM = ((T^i)_i, (M^i \times I^i, g^i)_i, (\Omega^i)_i, (U^i_{\pm})_i, (\Phi^i)_i)$ where each $g^i$ is a Ricci flow on the closed $3$-manifold $M^i$ defined for times $I^i$.
We can lift each of these flows to the universal cover $\widetilde{M}^i$ of $M^i$.
Its lift $\widetilde{g}^i$ still satisfies the Ricci flow equation.
Moreover, all its time slices are complete Riemannian metrics and we have bounded curvature on compact subintervals of $I^i$.
Denote by $\widetilde{\Omega}^i$ the preimage of $\Omega^i$ under the universal covering projection for each $i$.

We will now assemble the flows $(\widetilde{M}^i \times I^i, \widetilde{g}_t^i)$ to a Ricci flow with surgery.
Observe first that for every $i$, the subset $U^i_- \subset M^i$ is bounded by pairwise disjoint, embedded $2$-spheres.
So for every point $p \in U^i_-$, the natural map $\pi_1(U^i_-, p) \to \pi_1(M^i,p)$ is an injection.
Now let $\widetilde{U}^i_+ \subset \widetilde{M}^{i+1}$ be the preimage of $U^i_+$ under the universal covering projection.
The complement of this subset is still a collection of pairwise disjoint, embedded $3$-disks and hence $\widetilde{U}^i_+$ is simply connected.
Via $(\Phi^i)^{-1} : U^i_+ \to U^i_-$ there is a covering map $\widetilde{U}^i_+ \to U^i_+ \to U^i_- \subset M^i$.
Since $\td{U}^i_+$ is simply-connected, we find a lift $\phi : \widetilde{U}^i_+ \to \widetilde{M}^i$.
Using the fact that $U^i_- \to M^i$ is $\pi_1$-injective, we conclude that $\phi$ is injective.
Denote by $\widetilde{U}^i_- \subset \widetilde{M}^i$ the image of $\phi$ and let $\widetilde{\Phi}^i : \widetilde{U}^i_- \to \widetilde{U}^i_+$ be its inverse.
Then $\widetilde{\MM} =((T^i)_i, (\widetilde{M}^i \times I^i, \widetilde{g}_t^i)_i, (\widetilde{\Omega}^i)_i, (\widetilde{U}^i_{\pm})_i, (\widetilde{\Phi}^i)_i)$ is a Ricci flow with surgery.

The proof of Proposition \ref{Prop:Per73} can still be carried out for the Ricci flow with surgery $\widetilde{\MM}$ which possibly contains non-compact time slices.
Observe here that all time slices of $\td{\MM}$ are complete and the curvature is bounded on compact time intervals.
This gives us the desired result.
\end{proof}

\section{The thick-thin decomposition} \label{sec:thickthin}
We now describe how in the longtime picture Ricci flows with surgery decompose the manifold into a thick and a thin part.
In this process, the thick part approaches a hyperbolic metric while the thin part collapses on local scales.
Compare this Proposition with \cite[7.3]{PerelmanII} and \cite[Proposition 90.1]{KLnotes}.

\begin{Proposition} \label{Prop:thickthindec}
There is a function $\delta : [0, \infty) \to \IR_+$ such that given a Ricci flow with surgery and $\delta(t)$-precise cutoff $\MM$ with normalized initial conditions defined on the interval $[0,\infty)$, we can find a constant $T_0 < \infty$, a function $w : [T_0, \infty) \to \IR_+$ with $w(t) \to 0$ as $t \to \infty$ and a collection of orientable finite volume hyperbolic manifolds $(H'_1, g_{\hyp,1}), \ldots, (H'_k, g_{\hyp, k})$ such that: \\
There are finitely many embedded tori $T_{1,t}, \ldots, T_{m,t} \subset \MM(t)$ for $t \in [T_0, \infty)$ which move by isotopies and don't hit any surgery points and which separate $\MM(t)$ into two (possibly empty) closed subsets $\MM_{\thick}(t), \linebreak[1] \MM_{\thin}(t) \subset \MM(t)$ such that
\begin{enumerate}[label=(\textit{\alph*})]
\item $\MM_{\thick}(t)$ does not contain surgery points for all $t \in [T_0, \infty)$.
\item The $T_{i,t}$ are incompressible in $\MM(t)$ and $t^{-1/2} \diam_t T_{i,t} < w(t)$.
\item The topology of $\MM_{\thick}(t)$ stays constant in $t$ and $\MM_{\thick}(t)$ is a disjoint union of components $H_{1,t}, \ldots, H_{k,t} \subset \MM_{\thick}(t)$ such that the interior of each $H_{i,t}$ is diffeomorphic to $H'_i$.
\item We can find an embedded cross-sectional torus $T'_{j,t}$ in each cusp of the $H'_i$ which moves by isotopies such that the following holds:
Chop off the ends of the $H'_i$ along the $T'_{j,t}$ and call the remaining open manifolds $H''_{i,t}$.
Then each $H''_{i,t}$ contains a $w^{-1}(t)$-tubular neighborhood of the thick part\footnote{On the hyperbolic manifolds $H'_i$ the \emph{thick part} denotes the part in which the injectivity radius is larger than the Margulis constant.} of $H'_i$ and there are smooth families of diffeomorphisms $\Psi_{i,t} : H''_{i,t} \to H_i$ which become closer and closer to being isometries, i.e.
\[ \big\Vert \tfrac1t \Psi_{i,t}^* g(t) - g_{\hyp,i} \big\Vert_{C^{[w^{-1}(t)]}(H''_{i,t})} < w(t) \]
and which move slower and slower in time, i.e.
\[  \sup_{H''_{i,t}} t^{1/2} | \partial_t \Psi_{i,t} |  < w(t) \]
for all $t \in [T_0, \infty)$ and $i = 1, \ldots, k$.
\item A large neighborhood of the part $\MM_{\thin}(t)$ is better and better collapsed, i.e. for every $t \geq T_0$ and $x \in \MM(t)$ with 
\[ \dist_t (x, \MM_{\thin}(t)) < w^{-1}(t) \sqrt{t} \]
we have
\[ \vol_t B\big( x,t,\min\{ \rho(x,t), \sqrt{t} \} \big) < w(t) \big( \min \{ \rho(x,t), \sqrt{t} \} \big)^3. \]
\end{enumerate}
\end{Proposition}

\section{Analysis of the thin part}  \label{sec:MorganTian}
Based on property (e) of Proposition \ref{Prop:thickthindec} we can analyze the thin part $\MM_{\thin}(t)$ for large $t$ and recover its graph structure geometrically.
The following result follows from the work of Morgan and Tian (\cite{MorganTian}).
We have altered its phrasing to include more geometric information.
The reader can find an explanation below of where to find each of the following conclusions in their paper.
Similar results can also be found in (\cite{KLcollapse}), (\cite{BBMP2}), (\cite{CG}) and (\cite{Faessler}).

\begin{Proposition} \label{Prop:MorganTianMain}
For every two continuous functions $\ov{r}, K : (0, 1) \to (0, \infty)$ and every $\mu > 0$ there are constants $w_0 = w_0(\mu, \ov{r}, K) > 0$, $0 < s (\mu, \ov{r}, K) < \frac1{10}$ and $a(\mu) > 0$, monotone in $\mu$, such that:

Let $(M,g)$ be a Riemannian manifold and $M' \subset M$ a closed subset such that
\begin{enumerate}[label=(\textit{\roman*})]
\item Each component of $\partial M'$ is an embedded, incompressible torus of diameter $< w_0$ and around each such component there is a neighborhood in $M$ which is diffeomorphic to $T^2 \times I$ and which contains all points within distance less than $1$ and $- \frac{5}{16} \leq \sec \leq - \frac{3}{16}$ there
\item For all $x \in M'$ and $\rho_1 (x) = \min \{ \rho(x), 1 \}$ we have
\[ \vol B(x, \rho_1(x) ) < w_0 \rho_1^3(x). \]
\item For all $w \in (w_0, 1)$, $r < \ov{r}(w)$ and $x \in M'$ we have: if $\vol B(x,r) > w r^3$ and $r < \rho(x)$, then $|{\Rm}|, |\nabla \Rm |, | \nabla^2 \Rm | < K(w) r^{-2}$ on $B(x,r)$.
\end{enumerate}

Then either $M' = M$ and $M$ is diffeomorphic to an infra-nilmanifold or a manifold which also carries a metric of non-negative sectional curvature and $\diam M < \mu \rho_1(x)$ for all $x \in M$, or the following holds:

There are finitely many embedded $2$-tori $\Sigma^T_i$ and $2$-spheres  $\Sigma^S_i \subset M'$ which are pairwise disjoint and disjoint from $\partial M'$ as well as closed subsets $V_1, V_2, V'_2 \subset M'$ such that
\begin{enumerate}[label=(a\textit{\arabic*})]
\item $M' = V_1 \cup V_2 \cup V'_2$, the interiors of the sets $V_1, V_2$ and $V'_2$ are pairwise disjoint and $\partial V_1 \cup \partial V_2 \cup \partial V'_2 = \partial M' \cup \bigcup_i \Sigma^T_i \cup \bigcup_i \Sigma^S_i$.
Obviously, no two components of the same set share a common boundary.
\item $\partial V_1 = \partial M' \cup \bigcup_i \Sigma^T_i \cup \bigcup_i \Sigma^S_i$.
In particular, $V_2 \cap V_2' = \emptyset$ and $V_2 \cup V_2'$ is disjoint from $\partial M'$.
\item $V_1$ consists of components diffeomorphic to one of the following manifolds:
\[ T^2 \times I, \; S^2 \times I, \; \Klein^2 \widetilde{\times} I, \; \IR P^2 \widetilde{\times} I, \;  D^2 \times S^1, \; D^3, \]
a $T^2$ bundle over $S^1$, $S^2 \times S^1$ or the union of two (possibly different) components listed above along their $T^2$- or $S^2$-boundary.
\item Every component of $V_2'$ has exactly one boundary component and this component borders $V_1$ on the other side.
Moreover, every component of $V'_2$ is diffeomorphic to
\[ D^2 \times S^1, \; D^3, \; L(p,q) \setminus B^3, \; \Klein^2 \widetilde{\times} I. \]
\end{enumerate}

We can further characterize the components of $V_2$:
In $V_2$ we find embedded $2$-tori $\Xi^T_i$ and $\Xi^O_i$  which are pairwise disjoint and disjoint from the boundary $\partial V_2$.
Furthermore, there are embedded closed $2$-annuli $\Xi^A_i \subset V_2$ whose interior is disjoint from the $\Xi^T_i$, $\Xi^O_i$ and $\partial V_2$ and whose boundary components lie in the components of $\partial V_2$ which are spheres.
Each spherical component of $\partial V_2$ contains exactly two such boundary components which separate the sphere into two (polar) disks and one (equatorial) annulus $\Xi^E_i$.
We also find closed subsets $V_{2,\textnormal{reg}}, \linebreak[1] V_{2, \textnormal{cone}}, \linebreak[1] V_{2, \partial} \subset V_2$ such that
\begin{enumerate}[label=(b\textit{\arabic*})]
\item  $V_{2, \textnormal{reg}} \cup V_{2, \textnormal{cone}} \cup V_{2, \partial} = V_2$ and the interiors of these subsets are pairwise disjoint.
Moreover, $\partial V_{2, \textnormal{reg}}$ is the union of $\bigcup_i \Xi^T_i \cup \bigcup \Xi^O_i \cup \bigcup_i \Xi^A_i \bigcup_i \Xi^E_i$ and the components of $\partial V_2$ which are diffeomorphic to tori.
\item $V_{2, \textnormal{reg}}$ carries an $S^1$-fibration which is compatible with its boundary components and all its annular regions.
\item The components of $V_{2, \textnormal{cone}}$ are diffeomorphic to solid tori ($\approx D^2 \times S^1$) and bounded by the $\Xi^T_i$ such that the fibers of $V_{2, \textnormal{reg}}$ on their boundaries are not nullhomotopic inside $V_{2, \textnormal{cone}}$.
\item The components of $V_{2, \partial}$ are either solid tori and bounded by the $\Xi^O_i$ such that the $S^1$-fibers of $V_{2, \textnormal{reg}}$ on the $\Xi^O_i$ are nullhomotopic inside the $V_{2, \partial}$ or they are solid cylinders ($\approx D^2 \times I$) such that their two diskal boundary components are polar disks on $\partial V_2$ and their annular boundary component is one of the $\Xi^A_i$.
Every polar disk and every $\Xi^A_i$ bounds such a component on exactly one side.
\end{enumerate}

We now explain the geometric properties of this decomposition:
\begin{enumerate}[label=(c\textit{\arabic*})]
\item For every $x \in V_1$, the ball $(B(x,\rho_1(x)), \rho_1^{-1}(x) g, x)$ is $\mu$-close (in the Gromov-Hausdorff sense) to a $1$-dimensional interval $(J, g_{\eucl}, \ov{x})$ or an $S^1$ of length $> a(\mu)$.
In particular, if $x$ lies in a component of $V_1$ which is diffeomorphic to $T^2 \times I$ and does not border a component of $V'_2$, then $(J, \ov{x})$ can be chosen to be $( (-1,1), 0)$.

If $y \in B(x, \rho_1(x))$ is a point which is at least $\frac1{25}$ away from the endpoints of $J$ via this identification, then we can find an open subset $U$ with $B(y, \frac1{50} \rho_1(x)) \subset U \subset B(y, \frac1{25} \rho_1(x))$, a subinterval $J' \subset J$ and a map $p : (U, \rho_1^{-1}(x) g)  \to (J, g_{\eucl})$ such that
\begin{enumerate}
\item[({$\alpha$})] $p$ is $1$-Lipschitz and its differential has an eigenvalue $> 1- \mu$ everywhere,
\item[({$\beta$})] $U$ is diffeomorphic to $S^2 \times J'$ or $T^2 \times J'$ such that $p$ the projection map onto the interval,
\item[({$\gamma$})] the fibers of $p$ have diameter at most $\mu$.
\end{enumerate}
\item For every $x \in V_2$, the ball $(B(x,\rho_1(x)), \rho_1^{-1}(x) g, x)$ is $\mu$-close to a $2$-dimensional pointed Alexandrov space $(X, \ov{x})$ of area $> a$.
\item For every $x \in V_{2, \textnormal{reg}}$, the ball $(B(x, s \rho_1(x)), s^{-1} \rho_1^{-1} (x), x)$ is $\mu$-close to a standard $2$-dimensional Euclidean ball $(B = B_1(0), g_{\eucl}, \ov{x} = 0)$.

Moreover, there is an open subset $U$ with $B(x,\frac12 s \rho_1(x)) \subset U \subset \linebreak[1] B(x, \linebreak[0] s \rho_1 (x))$, a smooth map $p : U \to \IR^2$ such that:
\begin{enumerate}
\item[({$\alpha$})] there are vector fields $X_1, X_2$ on $U$ such that $dp (X_i) = \frac{\partial}{\partial x_i}$ and $X_1, X_2$ are almost orthonomal, i.e. $| \langle X_i, X_j \rangle - \delta_{ij} | < \mu$ for all $i,j = 1,2$,
\item[({$\beta$})] $U$ is diffeomorphic to $B^2 \times S^1$ such that $p : U \to p(U)$ corresponds to the projection onto $B^2$ and the $S^1$-fibers are isotopic to the fibers of the fibration on $V_{2, \textnormal{reg}}$.
\item[({$\gamma$})] the fibers of $p$ as well as the fibers of $V_{2, \textnormal{reg}}$ on $U$ have diameter at most $\mu$ and both families of fibers enclose an angle $< \mu$ with each other.
\end{enumerate}
\item For every $x \in V_{2, \textnormal{cone}}$, the ball $B(x, \frac1{10} \rho_1(x))$ covers the component of $V_{2, \textnormal{cone}}$ in which $x$ lies.
\end{enumerate}
\end{Proposition}

\begin{proof}
Conditions 1.-3. in \cite[Theorem 0.2]{MorganTian} follow from assumptions (i)-(iii) if we replace $w_0$ by a sequence $w_n \to 0$; except for the higher derivative bounds in 3. resp (iii) which are not really needed in the proof.
If $(\diam M) \rho_1^{-1} (x)$ is sufficiently small for some $x \in M$, then we can use \cite[Corollary 0.13]{FY} or arguments of the proof of \cite[Lemma 1.5]{MorganTian} to conclude that either $M$ carries a metric of non-negative curvature or it is an infra-nilmanifold.
So in the following we can assume that $(\diam M) \rho_1^{-1}(x) > \min \{c, \mu \} > 0$ for some universal $c > 0$ and drop Assumption 1 in \cite{MorganTian}.

Now using \cite[Theorem 1.1]{MorganTian}, we choose $V_1$ to be the set $V_{n,1}$ minus the $3$-balls whose closures were added in \cite[subsection 5.4.2]{MorganTian} and $V_2 \cup V'_2$ to be its complement in $M'$.
Let the set $V'_2$ consist of all such components which are near to $1$-dimensional spaces but which expand to be near $2$-dimensional spaces in the sense of \cite[Definition 5.4]{MorganTian}.
The topology of the components of $V'_2$ can be deduced using a better lower bound on the sectional curvature at the local scale.
For our purposes it will just be important that each of these components has only one boundary component.

By construction, for every $x \in V_1$, the ball $(B(x, \rho_1(x)), \rho_1^{-1}(x) g, x)$ is either $\varepsilon$-close to one of the required $1$-manifolds and sufficiently far away from boundary points or it lies in a product structure with $\varepsilon'$-control (see \cite[Propostion 5.2]{MorganTian}) or it is a component of type 2. as described in \cite[Lemma 5.3]{MorganTian}.
Note hereby that the constants $\varepsilon'$ and $\varepsilon$ can be chosen arbitrarily small (see subsection 4.8 in \cite{MorganTian}) and hence we can choose them for our purposes such that $\mu$-closeness is guaranteed in conclusion (c1).
The second part of (c1) follows from \cite[Lemma 4.38]{MorganTian}, possibly after reducing $\varepsilon$ and $\varepsilon'$ again.

For all points $x \in V_2$ we can conclude that $(B(x, \rho_1(x)), \rho_1^{-1}(x) g, x)$ is $\widehat{\varepsilon}$-close to a standard $2$-dimensional ball $\ov{B}$ of area $\geq a$ (see \cite[Proposition 5.2]{MorganTian}) or there is such a point $\frac1{10} \rho_1(x)$-close to $x$.
The second case only applies when $x$ lies very close to the boundary of $V_2$ and has to do with the fact that the set $V_2$ was constructed in several steps which involved removing and adding small collar neighborhoods.
Observe that if we decrease $a$ and increase $\widehat{\varepsilon}$ by a controlled amount, we can get rid of this second case.
Choosing $\widehat{\varepsilon} < \mu$ establishes (c2).
Observe hereby that the constant $a$ does not depend on the choice of $\widehat{\varepsilon}$ (see again subsection 4.8 in \cite{MorganTian}).

Now use \cite[Lemma 5.7]{MorganTian} and \cite[Theorem 3.22]{MorganTian} to construct $V_{2, \textnormal{reg}} = U_{2, \textnormal{generic}}$.
Furthermore, let $V_{2, \textnormal{cone}}$ be the set of all points in the complement near interior cone points and $V_{2, \partial}$ be the set of all points near flat $2$-dimensional boundary points or boundary corners.
Then, all points $x \in V_{2, \textnormal{reg}}$ are interior $\mu^{MT}$-flat on all scales $\leq s_2^{MT} \rho_1(x)$ (here $\mu^{MT}$ and $s_2^{MT}$ denote the constants from \cite{MorganTian}).
The constant $\mu^{MT}$ can be chosen arbitrarily small depending on $\widehat{\varepsilon}$.
This establishes the first part of (c3).
For the second part, we use \cite[Proposition 4.4]{MorganTian} to obtain a set $U$ together with an $S^1$-fibration structure.
Consider the base $B$ of this fibration, let $p : U \to B$ be the projection and equip $B$ with the submersion metric.
The sectional curvatures of this metric cannot be less than those on $U$ and the metric is Gromov-Hausdorff close to the metric on $U$ which in turn is Gromov-Hausdorff close to the Euclidean metric.
Hence, after possibly decreasing $U$ and $B$, we can find a coordinate system on $B$ for which $\frac{\partial}{\partial x_1}, \frac{\partial}{\partial x_2}$ are sufficiently close to being orthonormal.

Claim (c4) follows immediately.
\end{proof}

\section{Further geometric properties of the thin part} \label{sec:unwrap}
In this section we will identify parts in the decomposition of Proposition \ref{Prop:MorganTianMain} which become non-collapsed when we pass to the universal cover.

\begin{Lemma} \label{Lem:unwrapfibration}
There are constants $\mu_0, w_1 > 0$, where $w_1$ only depends on $s(\varepsilon, \linebreak[1]  \mu_0, \linebreak[1]  \ov{r}, \linebreak[1] K)$, such that:
Consider the situation of Proposition \ref{Prop:MorganTianMain} and assume $\mu \leq \mu_0$.
Let $x \in M'$ and consider one of the following cases:
\begin{enumerate}[label=(\roman*)]
\item $x \in \CC$ where $\CC$ is a component of $V_2$ with the property that the $S^1$-fiber of $\CC \cap V_{2, \textnormal{reg}}$ has infinite order in $\pi_1(M)$ or
\item $x \in \CC$ where $\CC$ is a component of $V_1$ which is diffeomorphic to $T^2 \times I$, not adjacent to any component of $V'_2$ and whose cross-sectional tori are incompressible in $M$.
\end{enumerate}
Then $\vol \td{B}(\widetilde{x}, r) \geq w_1 r^3$ for all $r \leq \rho_1(x)$ where $\td{B}(\widetilde{x}, r)$ denotes the $r$-ball around a lift $\widetilde{x}$ of $x$ in the universal cover $\widetilde{M}$ of $M$.
\end{Lemma}

We remark that Proposition stays true if in (2) we consider all components of $V_1$ whose generic fibers are incompressible tori.

\begin{proof}
By volume comparison, it suffices to prove the desired inequality only for $r = \rho_1(x)$.

\uline{Consider first case (i):}
Since the fibers on $\CC \cap V_{2, \textnormal{reg}}$ are non-contractible, we conclude that $\CC$ is disjoint from $V_{2, \partial}$.
So either $x \in V_{2, \textnormal{reg}}$ or $x \in V_{2, \textnormal{cone}}$.
In the second case, there is an $x' \in B(x, \frac1{10} \rho_1(x)) \cap V_{2, \textnormal{reg}}$ and $\rho_1(x') > \frac12 \rho_1(x)$.
Since $\td{B}(\td{x}', \frac12 \rho_1(x)) \subset \td{B}(\td{x}, \rho_1(x))$, we can replace $x$ by $x'$.
So assume without loss of generality that $x \in V_{2, \textnormal{reg}}$.

Consider now the map $p : U \to \IR^2$ and the metric $g' = s^{-1} \rho_1^{-1}(x) g$ on $M$.
For the rest of the proof of case (i) we will only work with the metric $g'$ on $M$ as opposed to $g$, and we will bound the $g'$-volume of a $1$-ball in the universal cover from below by a universal constant.
Observe that the sectional curvatures of the metric $g'$ are bounded from below by $-1$ on this ball.
In the following we will denote by $\delta_k(\mu_0)$ a positive constant, which depends on $\mu_0 > 0$ and which goes to zero as $\mu_0$ goes to zero.
We will then later choose $\mu_0$ small enough so that all constants $\delta_k$ are sufficiently small.
The following paragraphs carry out concepts which can also be found in \cite{BBI} or \cite{BGP}.

By the properties of $x$, we can find a $(2, \delta_1(\mu_0))$-strainer $(a_1, b_1, a_2, b_2)$ of size $\frac12$ around $x$ (here $\delta_1(\mu_0)$ is a suitable constant as mentioned above).
Recall that this means that for the comparison angle $\cangle$ in the model space of constant curvature $-1$ we have
\[ \cangle a_i x b_i > \pi - \delta_1, \quad \cangle a_i x b_j > \tfrac{\pi}2 - \delta_1, \quad \cangle a_i x a_j > \tfrac{\pi}2 - \delta_1, \quad \cangle b_i x b_j > \tfrac{\pi}2 - \delta_1 \]
and $\dist(a_i, x) = \dist(b_i, x) = \frac12$ for all $i \not= j$.
In the universal cover $\widetilde{M}$, we can now choose lifts $\widetilde{x}, \widetilde{a}_i, \widetilde{b}_i$ such that $\dist(\widetilde{a}_i, \widetilde{x}) = \dist(a_i, x) = \frac12$ and $\dist(\widetilde{b}_i, \widetilde{x}) = \dist(b_i, x) = \frac12$.
Since the universal covering map is $1$-Lipschitz, we obtain furthermore $\dist(\widetilde{a}_i, \widetilde{b}_j) \geq \dist(a_i, b_j)$, $\dist(\widetilde{a}_1, \widetilde{a}_2) \geq \dist(a_1, a_2)$ and $\dist(\widetilde{b}_1, \widetilde{b}_2) \geq \dist(b_1, b_2)$.
So all the comparison angles in the universal cover are at least as large as those on $M$ and hence we conclude that $(\widetilde{a}_1, \widetilde{b}_1, \widetilde{a}_2, \widetilde{b}_2)$ is a $(2, \delta_1(\mu_0))$-strainer around $\td{x}$ of size $\frac12$.

We now want to find a $2\frac12$-strainer around $\td{x}$.
To do this, observe that by the property of the map $p$ there is a sequence $\widetilde{x}_n$ of lifts of $x$ in $\widetilde{M}$ which is unbounded and whose consecutive distance is at most $2\mu_0$.
So we can find a point $\widetilde{y} \in \widetilde{M}$ with $\dist(\td{x},\td{y}) = 2 \sqrt{\mu_0}$ which projects to a point $y \in M$ with $\dist(x,y) < 2\mu_0$.
It follows that
\[ \dist(\widetilde{y}, \widetilde{a}_i) > \tfrac12 - 2 \mu_0 \qquad \text{and} \qquad \dist(\widetilde{y}, \widetilde{b}_i) > \tfrac12 - 2 \mu_0. \]
This implies
\begin{equation} \label{eq:212atx}
 \cangle \widetilde{y} \widetilde{x} \widetilde{a}_i > \tfrac{\pi}2 - \delta_2(\mu_0) \qquad \text{and} \qquad \cangle \widetilde{y} \widetilde{x} \widetilde{b}_i > \tfrac{\pi}2 - \delta_2(\mu_0).
\end{equation}
Hence $(\td{a}_1, \td{b}_1, \td{a}_2, \td{b}_2, \td{y})$ is a $2\frac12$-strainer around $\td{x}$.

Since $| \dist(\widetilde{y}, \widetilde{a}_i) - \dist(\widetilde{x}, \widetilde{a}_i)| < 2 \sqrt{\mu_0}$ and $| \dist(\widetilde{y}, \widetilde{b}_i) - \dist(\widetilde{x}, \widetilde{b}_i)| < 2 \sqrt{\mu_0}$, we conclude that $(\td{a}_1, \td{b}_1, \td{a}_2, \td{b}_2)$ is a $(2, \delta_3(\mu))$-strainer around $\td{y}$ of size $\geq \frac12 - 2 \sqrt{\mu_0}$.
We now show that symmetrically $(\td{a}_1, \td{b}_1, \td{a}_2, \td{b}_2, \td{x})$ is a $2\frac12$-strainer around $\td{y}$ of arbitrary good precision:
By comparison geometry
\[ \cangle \td{a}_i \td{x} \td{y} + \cangle \td{y} \td{x} \td{b}_i + \cangle \td{a}_i \td{x} \td{b}_i \leq 2 \pi. \]
Together with (\ref{eq:212atx}) and the strainer inequality at $\td{x}$, this yields
\[ \cangle \td{a}_i \td{x}  \td{y} < \tfrac{\pi}2 + \delta_1(\mu_0) + \delta_2(\mu_0). \]
By hyperbolic trigonometry,
\[ \cangle  \td{x} \td{y} \td{a}_i + \cangle \td{a}_i \td{x}  \td{y} + \cangle \td{y} \td{a}_i \td{x} > \pi - \delta_4(\mu_0) \qquad \text{and} \qquad \cangle \td{y} \td{a}_i \td{x}  < \delta_4(\mu_0). \]
Combining the last three inequalities yields
\[ \cangle \td{x} \td{y} \td{a}_i  > \tfrac{\pi}2 - \delta_1(\mu_0) - \delta_2(\mu_0) - 2 \delta_4(\mu_0) = \tfrac{\pi}2 - \delta_5(\mu_0). \]
The same estimate holds for $\cangle \td{x} \td{y} \td{b}_i$.

Let $\widetilde{m}$ be the midpoint of a minimizing segment joining $\widetilde{x}$ and $\widetilde{y}$.
We will now show that $(\td{a}_1, \td{b}_1, \td{a}_2, \td{b}_2, \td{y}, \td{x})$ is a $3$-strainer around $\td{m}$ of arbitrary good precision.
Since the distances of $\widetilde{a}_i$ resp. $\widetilde{b}_i$ to $\widetilde{m}$ differ from the distances to $\widetilde{x}$ by at most $\sqrt{\mu_0}$, we can conclude that $(\widetilde{a}_1, \widetilde{b}_1, \widetilde{a}_2, \widetilde{b}_2)$ is a $(2, \delta_6(\mu_0))$-strainer of size $\geq \frac12 - \sqrt{\mu_0}$ around $\td{m}$.
It remains to bound comparison angles involving the points $\td{x}$, $\td{y}$:
By the monotonicity of comparison angles, we have
\[ \cangle \td{m} \td{x} \td{a}_i \geq \cangle \td{y} \td{x} \td{a}_i > \tfrac{\pi}2 - \delta_2(\mu_0). \]
Now, if we apply the same argument as in the last paragraph, replacing $\td{x}$ with $\td{m}$, we obtain $\cangle \td{x} \td{m} \td{a}_i, \cangle \td{x} \td{m} \td{b}_i > \frac{\pi}2 - \delta_6(\mu_0)$.
For analogous estimates on the opposing angles, we then interchange the roles of $\td{x}$ and $\td{y}$.
Finally, $\cangle \td{x} \td{m} \td{y} = \pi$ is trivially true.

Set $\td{a}_3 = \td{y}$ and $\td{b}_3 = \td{x}$.
We have shown that $(\td{a}_1, \td{b}_1, \td{a}_2, \td{b}_2, \td{a}_3, \td{b}_3)$ is a $(3, \delta_7(\mu))$-strainer around $\td{m}$ of size $\geq \sqrt{\mu_0}$ for a suitable $\delta_7(\mu_0)$.
We will now use this fact to estimate the volume of the $\lambda \sqrt{\mu_0}$-ball around $\td{m}$ from below for sufficiently small $\lambda$ and $\mu_0$.
We follow here the ideas of the proof of \cite[Theorem 10.8.18]{BBI}.
Define the function
\begin{multline*}
 f : \td{B}(\td{m}, \lambda \sqrt{\mu_0}) \longrightarrow \IR^3 \qquad 
  z \longmapsto (\dist(\td{a}_1, z) - \dist(\td{b}_1, \td{m}), \\ \dist(\td{a}_2, z) - \dist(\td{b}_2, \td{m}), \dist(\td{a}_3, z) - \dist(\td{b}_3, \td{m})).
\end{multline*}
We will show that $f$ is $100$-bilipschitz for sufficiently small $\mu_0$ and $\lambda$.
Obviously, $f$ is $3$-Lipschitz, so it remains to establish the lower bound $\frac1{100}$.
Assume that this was false, i.e. that there are $z_1, z_2 \in \td{B}(\td{m}, \lambda \sqrt{\mu_0})$ with $\dist(z_1, z_2) > 100 | f(z_1) - f(z_2) |$.
Then for all $i = 1, 2, 3$
\begin{equation} \label{eq:almostiso}
 \dist(z_1, z_2) > 100 | {\dist(a_i, z_1) - \dist(a_i, z_2)} |.
\end{equation}
It is easy to see that given some $\delta > 0$, we can choose $\lambda > 0$ and $\mu_0 > 0$ sufficiently small, to ensure that $(\td{a}_1, \td{b}_1, \td{a}_2, \td{b}_2, \td{a}_3, \td{b}_3)$ is a $(3, \delta)$-strainer around $z_1$ and around $z_2$.
Now look at the comparison triangle corresponding to the points $z_1, z_2, \td{a}_i$.
By (\ref{eq:almostiso}), it is almost isosceles and hence by elementary hyperbolic trigonometry we conclude for $\lambda$ sufficiently small
\[ \tfrac{9}{10} \tfrac{\pi}2 < \cangle z_2 z_1 \td{a}_i, \; \cangle z_1 z_2 \td{a}_i < \tfrac{11}{10} \tfrac{\pi}2 . \]
Using comparison geometry
\[ \cangle z_1 z_2 \td{b}_i \leq 2 \pi - \cangle \td{a}_i z_2 \td{b}_i - \cangle z_1 z_2 \td{a}_i < \tfrac{11}{10} \tfrac{\pi}2 + \delta. \]
For $\lambda$ sufficiently small, we obtain furthermore by hyperbolic trigonometry
\[  \cangle \td{b}_i z_1 z_2 + \cangle z_1 z_2 \td{b}_i +\cangle z_2 \td{b}_i z_1 > \pi - \delta \qquad \text{and} \qquad \cangle z_2 \td{b}_i z_1 < \delta. \]
So
\[ \cangle \td{b}_i z_1 z_2 > \tfrac{9}{10} \tfrac{\pi}2 - 3 \delta. \]

Now join $z_1$ with $\td{a}_1, \td{b}_1, \td{a}_2, \td{b}_2, \td{a}_3$ by minimizing geodesics.
By comparison geometry, these geodesics enclose angles of at least $\frac{\pi}2 - \delta$ resp. $\pi - \delta$ between each other.
So their unit direction vectors approximate the negative and positive directions of an orthonormal basis.
By the same argument, the minimizing geodesic which connects $z_1$ with $z_2$ however encloses an angle of at least $\frac9{10} \frac{\pi}2 - 3\delta$ with each of these geodesics.
For sufficiently small $\delta$ this contradicts the fact that the tangent space at $z_1$ is $3$-dimensional.
So $f$ is indeed $100$-bilipschitz for sufficiently small $\lambda$ and $\mu_0$.

From the bilipschitz property we can conclude that
\[ \vol \td{B}(\td{m}, \lambda \sqrt{\mu_0} ) > c (\lambda \sqrt{\mu_0})^3 \]
for some universal $c > 0$.
Fixing $\mu_0 < \frac14$ and $\lambda < 1$ such that the argument above can be carried out, we obtain
\[ \vol \td{B}(\td{x}, 1) > \vol \td{B}(\td{m}, \lambda \sqrt{\mu_0}) > c (\lambda \sqrt{\mu_0})^3 = c' > 0. \]
By rescaling, this implies the desired inequality for the metric $g$.

\uline{Now consider case (ii):}
By Proposition \ref{Prop:MorganTianMain} we know that $(B(x, \rho_1(x)), \linebreak[1] \rho_1^{-1}(x) g, \linebreak[1] x)$ is $\mu$-close to $((-1,1), g_{\eucl}, 0)$ where $(-1,1)$ is an interval of length $2$.

Choose $q \in \pi_1(M)$ corresponding to a nontrivial simple loop in one of the cross-sectional tori and denote by $\widehat{M}$ the covering of $M$ corresponding to the cyclic subgroup generated by $q$, i.e. if we also denote by $q$ the deck-transformation of $\widetilde{M}$ corresponding to $q$, then $\widehat{M} = \td{M} / q$.
So we have a tower of coverings $\td{M} \to \widehat{M} \to M$.

Consider first the rescaled metric $g' = \rho_1^{-1} (x) g$.
Using the same arguments as in case (i), we can construct a $(1, \delta_1(\mu_0))$ strainer $(a_1, b_1)$ around $x$ on $M$ of size $\frac12$ for a suitable $\delta_1(\mu_0)$.
Furthermore, using the covering $\widehat{M} \to M$, we can find a point $\widehat{m} \in \widehat{M}$ within $\sqrt{\mu_0}$-distance away from a lift $\widehat{x}$ of $x$ and a $(2, \delta_2(\mu_0))$ strainer $(\widehat{a}_1, \widehat{b}_1, \widehat{a}_2, \widehat{b}_2)$ around $\widehat{m}$ of size $\geq \sqrt{\mu_0}$.
Connect the points $\widehat{a}_i$ and $\widehat{b}_i$ with $\widehat{m}$ by minimizing geodesics and choose points $\widehat{a}'_i$ and $\widehat{b}'_i$ of distance $\sqrt{\mu_0}$ from $\widehat{m}$.
By monotonicity of comparison angles, $(\widehat{a}'_1, \widehat{b}'_1, \widehat{a}'_2, \widehat{b}'_2)$ is a $(2, \delta_2(\mu_0))$-strainer of size $\sqrt{\mu_0}$.

Let $g'' = \frac12 \mu_0^{-1/2} g'$.
Then $(\widehat{a}'_1, \widehat{b}'_1, \widehat{a}'_2, \widehat{b}'_2)$ has size $\frac12$ with respect to $g''$.
Using this strainer, the metric $g''$ and the covering $\td{M} \to \widehat{M}$, we can apply the same argument from case (i) again and obtain a $(3, \delta_3(\mu_0))$ strainer $(\td{a}_1, \td{b}_1, \td{a}_2, \td{b}_2, \td{a}_3, \td{b}_3)$ around a point $\td{m}' \in \td{M}$ which is $\sqrt{\mu_0}$-close to a lift $\td{m}$ of $\widehat{m}$ in $\td{M}$.

As in case (i), for a sufficiently small $\mu_0$ we can deduce a lower volume bound $\vol_{g''} \td{B} (\td{m}', 1) > c'$.
With respect to $g'$, the point $\td{m}'$ is within a distance of $\mu_0 + \sqrt{\mu}_0$ of a lift $\td{x}$ of $\widehat{x}$.
Hence 
\[ \vol_{g'} \td{B}(\td{x}, 1) > \vol_{g'} \td{B}(\td{m}', 2 \sqrt{\mu}_0) > c' (2 \sqrt{\mu_0})^3 = c'' > 0. \]
The desired inequality follows by rescaling.
\end{proof}

\begin{Definition}
We call a component $\CC$ of $V_2$ resp. $V_1$ \emph{good} if it suffices the conditions in (i) resp. (ii) of Lemma \ref{Lem:unwrapfibration}.
\end{Definition}

\section{Evolution of areas of minimal surfaces} \label{sec:minsurf}
\begin{Lemma} \label{Lem:evolsphere}
Let $\MM$ be a Ricci flow with surgery and precise cutoff, defined on the time interval $[T_1, T_2]$ $(T_1 > 0)$, assume that the surgeries are all trivial and that $\pi_2 (\MM(t)) \not= 0$ for all $t \in [T_1, T_2]$.
For every time $t \in [T_1, T_2]$ denote by $A(t)$ the infimum of the areas of all homotopically nontrivial immersed $2$-spheres.
Then for all $t \in [T_1, T_2]$ we have $A(t) > 0$ and
\[ t^{-1} A(t) \leq T_1^{-1} A(T_1) - 4\pi (\log t - \log T_1) \]
\end{Lemma}
\begin{proof}
Compare also with \cite[Lemma 18.10 and 18.11]{MTRicciflow}.
Let $t_0 \in [T_1, T_2)$.
By \cite{SU81} and \cite{Gul} or \cite{MY}, there is a noncontractible, conformal, minimal immersion $f : S^2 \to \MM(t_0)$ with $\area_{S^2} f^* (g(t_0)) = A(t_0)$.
Call $\Sigma = f(S^2) \subset \MM(t)$.
We can estimate the infinitesimal change of its area while we vary the metric in positive time direction (and keep $f$ constant!).
Using the fact that the interior sectional curvatures are not larger than the ambient ones as well as Gau\ss-Bonnet, we conclude:
\begin{multline*}
\frac{d}{dt^+} \Big|_{t = t_0} \area_{t} (\Sigma ) = - \int_{\Sigma} \tr_{t_0} (\Ric_{t_0} |_{T \Sigma}) d {\vol}_{t_0} \\
= - \frac12 \int_{\Sigma} R_{t_0} d {\vol}_{t_0} - \int_{\Sigma} \sec_{t_0}^{\MM(t_0)}(T \Sigma) d {\vol}_{t_0} 
 \leq \frac3{4t_0} \area_{t_0} (\Sigma) - \int_{\Sigma} \sec^\Sigma d {\vol}_{t_0} \\
 \leq \frac3{4t_0} \area_{t_0}(\Sigma) - 2 \pi \chi(\Sigma) = \frac3{4t_0} A(t_0) - 4 \pi.
\end{multline*}
Here, $\sec^{\MM(t_0)}_{t_0}(T\Sigma)$ denotes the ambient sectional curvature of ${\MM(t_0)}$ tangential to $\Sigma$ and $\sec^\Sigma_{t_0}$ denotes the interior sectional curvature of $\Sigma$.
We conclude from this calculation that $\frac{d}{dt^+}|_{t = t_0} (t^{-1} A(t) + 4 \pi ( \log t - \log T_1)) <0$ in the barrier sense and hence, the function $A(t) + 4 \pi (\log t - \log T_1)$ is monotonically decreasing in $t$ away from the singular times.

We will now show that $A(t)$ is lower semi-continuous.
We can restrict ourselves to the case in which $t_0$ is a surgery time.
Let $t_k \nearrow t_0$ be a sequence converging to $t_0$ and choose minimal $2$-spheres $\Sigma_k \subset \MM(t_k)$ with $\area_{t_k} \Sigma_k = A(t_k)$.
By property (5) of Definition \ref{Def:precisecutoff}, we find diffeomorphisms $\xi_k : \MM(t_k) \to \MM(t_0)$ which are $(1+\chi_k)$-Lipschitz for $\chi_k \to 0$.
So $A(t_0) \leq \lim \inf_{k \to \infty} (1+\chi_k)^2 A(t_k) = \lim \inf_{k \to \infty} A(t_k)$.
\end{proof}

\begin{Lemma} \label{Lem:areaofhomotopy}
Let $\MM$ be a Ricci flow with surgery and precise cutoff, defined on the time interval $[T_1, \infty)$ $(T_1 \geq 0)$ and assume that the surgeries are all trivial.
Let $\gamma_{1,t}, \gamma_{2,t} \subset \MM(t)$ be two families of smoothly embedded noncontractible loops which are homotopic to each other and move by isotopies for all $t \in [T_1,\infty)$.
For every $t \in [T_1, \infty)$ let $A(t)$ be the infimum over the areas of all smooth homotopies $S^1 \times I \to \MM(t)$ connecting $\gamma_{1,t}$ with $\gamma_{2,t}$.

Assume that for the geodesic curvatures we have the bound $\kappa(\gamma_{1,t}), \kappa(\gamma_{2,t}) < C t^{-1}$ for all $t \in [T_1, \infty)$ and assume that the normalized lengths $t^{-1/2} \ell(\gamma_{1,t})$, $t^{-1/2} \ell(\gamma_{2,t})$ converge to $0$ as $t \to \infty$.
Moreover, assume that the velocity by which the given loops move, is bounded in the appropriate rescaling, i.e. $|\partial_t \gamma_{1,t}|, \linebreak[1] |\partial_t \gamma_{2,t}| < C t^{-1/2}$ for all $t \in [T_1, \infty)$.

Then $t^{-1} A(t) \to 0$ as $t \to \infty$.
\end{Lemma}

\begin{proof}
Let $t_0 \in [T_1, \infty)$.
By \cite{Mor}, we can find an area minimizing homotopy between $\gamma_{1,t_0}$ and $\gamma_{2, t_0}$.
More precisely, there is an $0 < r < 1$ such that if we denote by $A_{r,1} = \ov{B}_1(0) \setminus B_r(0) \subset \IC$ the closed $(r,1)$-annulus, then we can find a continuous map $f : A_{r,1} \to \MM(t_0)$ with the following properties: $f$ restricted to the boundary components of $A_{r,1}$ represents a parameterization of $\gamma_{1,t_0}$ resp. $\gamma_{2,t_0}$. Moreover, $f$ is smooth, conformal and harmonic on the interior of $A_{r,1}$ and we have $A(t_0) = \area f^*(g(t_0))$.
By \cite{HH}, $f$ is even smooth up to the boundary.

Analogously to the proof of Lemma \ref{Lem:evolsphere}, we can compute the infinitesimal change of the area of $f$ as we vary the metric only:
\begin{multline*}
 \frac{d}{dt} \Big|_{t = t_0} \area f^*(g(t)) = - \int_{A_{r,1}} \tr f^*(\Ric^{\MM(t_0)}_{t_0}) \\
 \leq \frac{3}{4t_0} A(t_0) - \int_{A_{r,1}}  \sec^{\MM(t_0)}( df) d {\vol}_{f^* ( g(t_0) )},
\end{multline*}
where $\sec^{\MM(t_0)} (df)$ denotes the sectional curvature in the normalized tangential direction of $f$. 
Observe that the last integrand is a continuous function on $A_{r,1}$ since the volume form vanishes wherever this tangential sectional curvature is not defined.

In order to avoid issues arising from possible branch points (especially on the boundary of $A_{r,1}$), we employ the following trick (compare with \cite{PerelmanIII}):
Let $\varepsilon > 0$ be a small constant and consider the flat cylinder $(N_\varepsilon = S^1 \times [\log r, 0]), \varepsilon ( g_{S^1} + g_{\eucl}))$ of size $\varepsilon$.
Then $h_\varepsilon : A_{r,1} \to N_\varepsilon, z \mapsto (\log |z|, z |z|^{-1})$ is a conformal and harmonic diffeomorphism.
We conclude that the map $f_\varepsilon = (f, h_\varepsilon) : A_{r,1} \to \MM(t_0) \times N_\varepsilon$ is a conformal and harmonic \emph{embedding}.
Denote its image by $\Sigma_\varepsilon = f_\varepsilon(A_{r,1})$.
Since the sectional curvatures on the target manifold are bounded, we have
\[ \lim_{\varepsilon \to 0} \int_{\Sigma_\varepsilon} \sec^{\MM(t_0) \times N_\varepsilon} ( T \Sigma_\varepsilon) d {\vol}_{t_0} =  \int_{A_{r,1}}  \sec^{\MM(t_0)}( df) d {\vol}_{f^*(g(t_0))}. \]
We can now proceed as in the proof of Lemma \ref{Lem:evolsphere}, using the fact that the interior sectional curvatures of $\Sigma_\varepsilon$ are not larger than the corresponding ambient ones as well as Gau\ss-Bonnet:
\[ \int_{\Sigma_\varepsilon} \sec^{\MM(t_0) \times N_\varepsilon} ( T \Sigma_\varepsilon) d {\vol}_{t_0} \geq \int_{\Sigma_\varepsilon} \sec^{\Sigma_\varepsilon} ( T \Sigma_\varepsilon) d {\vol}_{t_0} = 2 \pi \chi(\Sigma_\varepsilon) + \int_{\partial \Sigma_\varepsilon} \kappa^{\Sigma_\varepsilon}_{\partial \Sigma_\varepsilon} d s_{t_0}. \]
In our case $\chi(\Sigma_\varepsilon) = 0$.
We now estimate the last integral.
Let $\gamma_{1\text{ resp. }2,\varepsilon} : S^1(l_{1 \text{ resp. } 2, \varepsilon}) \to \partial \Sigma_{\varepsilon}$ be unit-speed parameterizations of the boundary of $\Sigma_\varepsilon$.
Denote by $\gamma^{\MM(t_0)}_{1\text{ resp. }2, \varepsilon}(s)$ their component functions in $\MM(t_0)$.
Furthermore, let $\nu_{1\text{ resp. }2, \varepsilon}(s)$ be the outward-pointing unit-normal field along $\gamma_{1\text{ resp. }2, \varepsilon}(s)$ which is tangent to $\Sigma_{\varepsilon}$.
It is not difficult to see that due to conformality, the $\MM(t_0)$-component of $\nu_{1\text{ resp. }2, \varepsilon}(s)$ has the same length as the component of the velocity vector $(\gamma^{\MM(t_0)}_{1\text{ resp. }2, \varepsilon})' (s)$ at that point.
Hence, since the boundary of $N_\varepsilon$ is geodesic, we can compute
\[ - \int_{\partial \Sigma_\varepsilon} \kappa^{\Sigma_\varepsilon}_{\partial \Sigma_\varepsilon} d s_{t_0} = - \sum_{i = 1,2} \int_0^{l_i} \Big\langle \frac{D}{ds} \Big( \frac{d}{ds}  \gamma^{\MM(t_0)}_{i, \varepsilon} (s) \Big), \nu^{\MM(t_0)}_{i, \varepsilon}(s) \Big\rangle ds \leq C t_0^{-1/2} (l_{1, \varepsilon} + l_{2, \varepsilon}). \]
Passing to the limit $\varepsilon \to 0$ and using $l_{1/2, \varepsilon} \to \ell(\gamma_{1/2, t_0})$, we hence obtain
\[ \frac{d}{dt} \Big|_{t = t_0} \area f^* ( g(t)) \leq \frac{3}{4 t_0} A(t_0) + C t_0^{-1/2} \big( \ell(\gamma_{1,t_0}) + \ell(\gamma_{2, t_0}) \big). \]

In order to bound the derivative of $A(t)$ in the barrier sense, we have to account for the fact that the boundary curves move by isotopies.
The maximal additional infinitesimal increase is then
\[ \ell(\gamma_{1,t_0}) \sup_{\gamma_{1,t_0}} | \partial_t \gamma_{1,t_0} | + \ell(\gamma_{2,t_0}) \sup_{\gamma_{2,t_0}} | \partial_t \gamma_{2,t_0}| \leq C t_0^{-1/2} \big(\ell(\gamma_{1,t_0}) + \ell( \gamma_{2,t_0}) \big). \]
So in the barrier sense
\[ \frac{d}{dt^+} \Big|_{t = t_0} A(t) \leq \frac{3}{4 t_0} A(t_0) + 2 C t_0^{-1/2} \big( \ell(\gamma_{1,t_0}) + \ell(\gamma_{2,t_0}) \big). \]
Thus
\begin{equation} \label{eq:tAtbarrier}
 \frac{d}{dt^+} \big( t^{-1} A(t) \big) \leq - \frac{1}{t} \big( \tfrac{1}{4} \big( t^{-1} A(t) \big) - 2 C t^{-1/2} \big( \ell(\gamma_{1,t}) + \ell(\gamma_{2,t}) \big) \big).
\end{equation}

Analogously as in the proof of Lemma \ref{Lem:evolsphere}, we conclude that $A(t)$ is lower semi-continuous.
So since the last summand in (\ref{eq:tAtbarrier}) goes to $0$ for $t \to \infty$, we conclude that for every $a > 0$ there is some time $t_1$ such that whenever $t \geq t_1$ and $t^{-1} A(t) \geq a$, then $\frac{d}{dt^+} (t^{-1} A(t)) < - \frac18 t^{-1} a$.
Since $t^{-1}$ is not integrable, this implies that $t^{-1} A(t) < a$ for large $t$.
It follows that $t^{-1} A(t) \to 0$ as $t \to \infty$.
\end{proof}

\section{Proof of Theorem \ref{Thm:main}} \label{sec:MainThm}
In order to finish the proof of the main theorem, we will need the following topological statement, which will help us to assure that minimal annuli pass through certain thin parts of the manifold.
\begin{Lemma} \label{Lem:hastointersectannulus}
Let $M$ be a smooth closed $3$-manifold and $U_1, \ldots, U_m \subset M$ pairwise disjoint embedded copies of $T^2 \times I$ such that the components of $M'' = M \setminus (U_1 \cup \ldots \cup U_m)$ are hyperbolic (i.e. they carry hyperbolic metrics of finite volume).
Let $\sigma_1, \sigma_2 : S^1 \to M''$ be two loops which are freely homotopic to each other in $M$, but which lie in different components $M_1$ resp. $M_2$ of $M''$ which are both adjacent to $U_1$.
Moreover, assume that $\sigma_1, \sigma_2$ are freely homotopic in $M_1$ resp. $M_2$ to nontrivial loops $\sigma'_1, \sigma'_2$ in each boundary torus of $U_1$.

Then the image of every homotopy $f : S^1 \times I \to M$ between $\sigma_1$ and $\sigma_2$ has to intersect every loop $\gamma \subset U_1$ which is not homotopic to a multiple of $\sigma'_1$ resp. $\sigma'_2$ in $U_1$.
\end{Lemma}
\begin{proof}
Assume that some $\gamma \subset U_1$ does not intersect $f(S^1 \times I)$.
By a perturbation argument, we can assume that $f$ is transverse to the boundaries of all the $U_i$.
So $f^{-1}(\partial U_1 \cup \ldots \cup \partial U_m)$ is a collection of disjoint circles $C_1, \ldots, C_k \subset S^1 \times I$.
If one of these circles is contractible in $S^1 \times I$, then pick an innermost contractible circle $C_j$.
It bounds a disk $D_j$.
The image of its interior has to be contained in one of the components of $M''$ or in one of the $U_i$.
In either case, this implies that $f |_{C_j}$ is homotopically trivial in the corresponding boundary torus and hence, we can replace $f$ by a transverse $f'$ which intersects the boundaries of the $U_i$ in one circle less and whose image still does not meet $\gamma$.

So after a finite number of reduction steps, we can assume that all the $C_j$ are noncontractible in $S^1 \times I$, which implies that they cut this annulus into $k-1$ nested topological annuli.
Assume that one of these annuli is bounded by loops $C_{j_1}$ and $C_{j_2}$ such that the image of $C_{j_1}$ is contained in $\partial U_1$, but the image of $C_{j_2}$ is contained in some $\partial U_i$ with $i \not= 1$.
This means that two cuspidal homotopy classes of $M_1$ or $M_2$ which correspond to different cusps, are conjugate to each other.
However, this is impossible by elementary hyperbolic geometry.

So the images of all $C_j$ must be contained in $\partial U_1$.
By elementary hyperbolic geometry again, we conclude that $f|_{C_j}$ is homotopic to $\sigma'_1$ resp. $\sigma'_2$ in $\partial U_1$.
So if we restrict $f$ to a certain sub-annulus, we obtain a homotopy $f'' : S^1 \times I \to U_1 \approx T^2 \times I$ between loops in each boundary torus which are each homotopic to $\sigma'_1$ resp. $\sigma'_2$ in $\partial U_1$.

By a simple intersection number argument, the image of $f''$ has to intersect $\gamma$.
\end{proof}

We now prove that after some large time, all time slices are irreducible and all surgeries are trivial (see also \cite[Proposition 18.9]{MTRicciflow}).
\begin{Proposition} \label{Prop:irreducibleafterfinitetime}
Let $\MM$ be a Ricci flow with surgery and precise cutoff, defined on the time interval $[T, \infty)$ $(T\geq0)$.
Then there is some $T_1 \in [T, \infty)$ such that for all $t \in [T_1, \infty)$, the time slice $\MM(t)$ only consists irreducible aspherical components, $\MM(t) \approx \MM(T_1)$ and all surgeries on $[T_1, \infty)$ are trivial.
\end{Proposition}
\begin{proof}
By definition of Ricci flows with surgery, for any two times $t_2 > t_1 \geq T$, the topological manifold $\MM(t_1)$ can be obtained from $\MM(t_2)$ by possibly adding spherical space forms or copies of $S^2 \times S^1$ to the components of $\MM(t_2)$ and then performing connected sums between some components.
So by the existence and uniqueness of the prime decomposition (see e.g. \cite[Theorem 1.5]{Hat}), there are only finitely many times when the topology of $\MM(t)$ can change.
This implies that there is some $T_1 \in [T, \infty)$ such that the time slices $\MM(t)$ are diffeomorphic to each other for all $t \in [T_1, \infty)$.

By finite time extinction of spherical components (see \cite{PerelmanIII}, \cite{ColdingMinicozziextinction}), we conclude that $\MM(t)$ cannot have components which are spherical space forms for $t \in [T_1, \infty)$.

It remains to prove that all components of $\MM(t)$ are irreducible.
Assume not.
Then by \cite[Proposition 3.10]{Hat} and the solution of the Poincar\'e Conjecture, we find that $\pi_2(N) \not= 0$ for some component $N$ of $\MM(t)$.
We can now use Lemma \ref{Lem:evolsphere} and conclude that $t^{-1} A(t)$ goes to zero in finite time.
This is a contradiction to the fact that $A(t) > 0$.
\end{proof}

\begin{proof}[Proof of Theorem \ref{Thm:main}]
First, by Proposition \ref{Prop:irreducibleafterfinitetime} and the assumption of the Theorem we conclude that for all $t \geq T_1$ the (topological) manifold $\MM(t)$ consists of components which are irreducible and only contain hyperbolic pieces in their torus decomposition.
Moreover, all surgeries on $[T_1, \infty)$ are trivial.
 
Next, we apply Proposition \ref{Prop:thickthindec} (here we need to assume that $\MM$ is performed by sufficiently precise cutoff).
This yields, amongst others, a time $T_2 > T_1$, a splitting $\MM(t) = \MM_{\thick}(t) \cup \MM_{\thin}(t)$ for all $t \in [T_2, \infty)$ and a function $w : [T_2, \infty) \to \IR_+$ with $w(t) \to 0$ as $t \to \infty$.
The interiors of the components of $\MM_{\thick}(t)$ are diffeomorphic to the hyperbolic manifolds $H'_1, \ldots, H'_k$ and $\MM_{\thin}(t)$ satisfies the collapsing condition described in Proposition \ref{Prop:thickthindec}(e).
Moreover, the components of $\MM_{\thick}(t)$ and $\MM_{\thin}(t)$ are separated by embedded, incompressible tori $T_{1,t}, \ldots, T_{m,t} \subset \MM(t)$.

Choose $\mu = \mu_0$ from Lemma \ref{Lem:unwrapfibration} and then $w_0 = w_0 ( \mu, \ov{r}, K)$ from Proposition \ref{Prop:MorganTianMain}, where $\ov{r}$ and $K$ are the functions from Proposition \ref{Prop:Per73} (in order to apply this Proposition, we again have to assume that the surgeries of $\MM$ are performed by sufficiently precise cutoff).
We can find some time $T_3 > T_2$ such that $w(t) < w_0$ for all $t \in [T_3, \infty)$ and hence Proposition \ref{Prop:MorganTianMain} can be applied to $\MM_{\thin}(t)$ for $\mu = \mu_0$, which gives us a decomposition of the thin part.

We now need to prove that for all $t \in [T_3, \infty)$, all components of $\MM_{\thin}(t)$ are diffeomorphic to $T^2 \times I$:
By \cite[Theorem 0.2]{MorganTian} (observe that this Theorem is a direct consequence of Proposition \ref{Prop:MorganTianMain}), we can choose additional embedded, incompressible tori $T'_{1,t}, \ldots, T'_{m',t} \subset \MM_{\thin}(t)$ which cut $\MM_{\thin}(t)$ into Seifert pieces.
Using the uniqueness of the torus decomposition (see \cite[Theorem 1.9]{Hat}) and the topological assumption on $\MM(t)$, we conclude that a subset $\mathcal{T} \subset \mathcal{T}_0 = \{ T_{1,t}, \linebreak[1] \ldots, T_{m, t}, T'_{1,t}, \ldots, T'_{m',t} \}$ cuts $\MM(t)$ into pieces which are hyperbolic.
Let $H \subset \MM(t) \setminus \mathcal{T}$ be such a hyperbolic piece and consider a torus $T \in \mathcal{T}_0 \setminus \mathcal{T}$ which is contained in $H$.
Since hyperbolic manifolds are atoroidal, there is a boundary torus $T'' \in \mathcal{T}$ of $H$ such that $T$ and $T''$ bound an embedded copy of $T^2 \times I$.
We conclude that the tori of $\mathcal{T}_0$ which are contained in $H$, cut $H$ into pieces which are diffeomorphic to $T^2 \times I$ except for one piece which is diffeomorphic to $H$.
Since $H$ cannot carry a Seifert structure, this piece cannot be contained in $\MM_{\thin}(t)$.
So $\MM_{\thin}(t) \setminus \mathcal{T}_0$ is a disjoint union of copies of $T^2 \times I$.
Piecing these together, implies that all components of $\MM_{\thin}(t)$ are diffeomorphic to $T^2 \times I$.

Having established the topological description, we will now try bound the geometry of the thin part using a minimal surface argument.
In order to do that, we choose smooth isotopies of loops ${\sigma'}^1_{1, t}, {\sigma'}^2_{1, t}, \ldots, {\sigma'}^1_{m, t}, {\sigma'}^2_{m, t} : S^1 \to H'_1 \cup \ldots \cup H'_k$ in the model hyperbolic manifolds, defined for times $t \in [T_3, \infty)$ such that there is a function $\varepsilon : [T_3, \infty) \to \IR_+$ with $\varepsilon(t) \to 0$ as $t \to \infty$ and:
\begin{enumerate}
\item The lengths of the loops go to zero: $\ell({\sigma'}^j_{i,t}) < \varepsilon(t)$ for all $t \in [T_3, \infty)$ and their geodesic curvature is everywhere equal to $1$.
\item For all $t$, the loops ${\sigma'}^j_{i,t}$ are contained in $H''_{1,t} \cup \ldots \cup H''_{k, t} \subset H'_1 \cup \ldots \cup H'_k$ (compare with Proposition \ref{Prop:thickthindec}(d)).
\item The velocity by which the loops move, is bounded appropriately: $| \partial_t {\sigma'}^j_{i, t} | < t^{-1}$.
\item For every hyperbolic cusp $N' \subset H'_1 \cup \ldots \cup H'_k$, consider the torus $T_{i,t} \subset \MM(t)$ which borders the corresponding almost hyperbolic cusp $N \subset \MM(t)$.
Then ${\sigma'}^1_{i,t}, {\sigma'}^2_{i,t}$ are contained in $N'$ and for all $t \in [T_3, \infty)$ and represent two nondivisible and linearly independent homotopy classes in $\pi_1(N') \cong \pi_1(T^2 \times I) \cong \IZ^2$.
\item Let now $\sigma^j_{i,t} : S^1 \to \MM(t)$ be the loops corresponding to the ${\sigma'}^j_{i,t}$ under the diffeomorphisms $\Psi_{l,t} : H''_{l,t} \to H_{l,t}$, i.e. $\sigma^j_{i,t} = \Psi_{l,t} \circ {\sigma'}^j_{i,t}$ for the appropriate $l$ (see Proposition \ref{Prop:thickthindec}(d)).
We now demand that for every component $\CC \subset \MM_{\thin}(t)$ the following is true: let $N_1, N_2 \subset H_1 \cup \ldots \cup H_k$ be the two cusps which are adjacent to $\CC$ and let $\sigma^1_{i_1, t}, \sigma^2_{i_2, t}$ be the loops in $N_1$ and $\sigma^1_{i_2, t}, \sigma^2_{i_2, t}$ the loops in $N_2$.
Then $\sigma^1_{i_1, t}$ and $\sigma^1_{i_2, t}$ resp. $\sigma^2_{i_1, t}$ and $\sigma^2_{i_2, t}$ are freely homotopic in $\MM(t)$.
\end{enumerate}
It is clear that we can find such $\sigma^j_{i,t}$, e.g. by choosing the loops as geodesics of horospherical tori in the cusps, $d(t)$-far away from the thick part, where $d(t)$ is an interpolation of $\min \{ w^{-1}(t), \log t \}$.

For each time $t \in [T_3, \infty)$ and component $\CC \subset \MM_{\thin}(t)$ denote by $A_{\CC, j}(t)$ the infimum over the areas of all smooth homotopies $S^1 \times I \to \MM(t)$ connecting $\sigma^j_{i_1, t}$ and $\sigma^j_{i_2, t}$ from property (5).
By Lemma \ref{Lem:areaofhomotopy} and conditions (1)--(3) above, we conclude that $t^{-1} A_{\CC, j}(t) \to 0$ as $t \to \infty$.
So there are time-dependent homotopies $f^j_{\CC, t} : S^1 \times I \to \MM(t)$ such that
\begin{equation} \label{eq:CChomotopy0}
 t^{-1} \area_t f^j_{\CC, t} \longrightarrow 0 \qquad \text{as} \qquad t \longrightarrow \infty
\end{equation}
for all components $\CC$ of $\MM_{\thin}(t)$ and $j = 1,2$. (Note that the components $\CC$ change in time.
However, the combinatorics of the thick-thin decomposition stay the same on $[T_3, \infty)$.)

Now look at the decomposition of a component $\CC \subset \MM_{\thin}$ into sets $V_1$, $V_2$, $V'_2$ as given in Proposition \ref{Prop:MorganTianMain} (applied to the metric $t^{-1} g(t)$).
The two boundary tori of $\CC$ have to border components of $V_1$.
So either $\CC = V_1$ or the boundary components of $\CC$ border components $\CC_1, \CC_2 \subset V_1$ which are diffeomorphic to $T^2 \times I$ (see conclusion (a3)).
In the second case, there is a component $\CC_3$ of $V_2$ or $V'_2$ adjacent to $\CC_1$.
Since components of $V'_2$ have only one boundary component and $\CC \not= \CC_1 \cup \CC_3$, we must have $\CC_3 \subset V_2$.
The generic $S^1$-fibers of $\CC_3$ are homotopic to a nontrivial curve in the boundary torus of $\CC_1$ adjacent to $\CC_3$.
This torus is isotopic to one of the $T_{i,t}$ which are incompressible in $\MM(t)$ (see Proposition \ref{Prop:thickthindec}(b)).
So the generic $S^1$-fibers of $\CC_3$ generate an infinite cyclic subgroup in $\pi_1(\MM(t))$.

Hence, we can apply Lemma \ref{Lem:unwrapfibration} and obtain that for any $x \in \CC$ (if $\CC = V_1$) or for any $x \in \CC_1 \cup \CC_2 \cup \CC_3$ (if $\CC \not= V_1$), we have $\vol_t \td{B}(\td{x}, \rho_{\sqrt{t}}(x,t)) \geq w_1 \rho_{\sqrt{t}}^3(x,t)$ in $\td{\MM}(t)$.
We can now use Proposition \ref{Prop:Per73univcover}, to deduce that there is some $T_4 \in [T_3, \infty)$ and constants $\ov{\rho} > 0$ and $K < \infty$ such that for all $t \in [T_4, \infty)$ we have $\rho(x,t) > \ov{\rho} \sqrt{t}$ and $|{\Rm}| < K t^{-1}$ on $\CC$ (if $\CC = V_1$) resp. $\CC_1 \cup \CC_2 \cup \CC_3$ (if $\CC \not= V_1$). 

Assume that the second case occurs for some $t \in [T_4, \infty)$.
Let $x \in C_3$.
Then by Proposition \ref{Prop:MorganTianMain}(c3), we can find an open set $U$ with $B(x, t, \frac12 s \rho_{\sqrt{t}}(x,t)) \subset U \subset B(x, t, s \rho_{\sqrt{t}}(x,t))$ and a $2$-Lipschitz map $p : U \to \IR^2$ whose image must contain $B(0, \frac14 s \rho_{\sqrt{t}}(x,t)) \subset \IR^2$ and whose fibers are homotopic to the fibers on $\CC_3$ and hence non-contractible in $\MM(t)$.
So by Lemma \ref{Lem:hastointersectannulus} applied twice, we conclude that each fiber of $p$ has to intersect the images of one of the homotopies $f^1_{\CC, t}, f^2_{\CC, t}$.
This implies that
\[ \area_t f^1_{\CC, t} + \area_t f^2_{\CC,t} > c s^2 \rho^2_{\sqrt{t}}(x,t) > c s^2 \ov{\rho}^2 t \]
for some universal $c > 0$.
If $t$ is sufficiently large, this however contradicts (\ref{eq:CChomotopy0}).

We conclude that there is some $T_5 \in [T_4, \infty)$ such that for all $t \in [T_5, \infty)$, we have $\CC = V_1$ for all components $\CC \subset \MM_{\thin}(t)$ and $|{\Rm}| < K t^{-1}$ on $\MM_{\thin}(t)$.
The curvature bound on $\MM_{\thick}(t)$ follows directly from Proposition \ref{Prop:thickthindec}(d).
By Definition \ref{Def:precisecutoff}(3), surgeries can only appear when the curvature is comparable to $\delta^{-2}(t)$, where $\delta(t)$ is the preciseness parameter.
So if we assume that $\MM$ is performed by sufficiently precise cutoff, then there cannot be any surgeries for large $t$.
\end{proof}

\end{document}